\documentclass[11pt]{amsart}\usepackage{cases}\usepackage{amsmath}\usepackage{amssymb}\usepackage{amsfonts}
\usepackage{amsthm}\usepackage{latexsym}\usepackage{graphicx}\usepackage{epsfig}
\newtheorem{theorem}{Theorem}[section]
\newtheorem{lemma}[theorem]{Lemma}
\newtheorem{corollary}[theorem]{Corollary}
\newtheorem{proposition}[theorem]{Proposition}

\theoremstyle{definition}

\newtheorem{example}[theorem]{Example}
\newtheorem{remark}[theorem]{Remark}
\numberwithin{equation}{section}

\begin{document}
\title[Heat kernel on smooth metric measure spaces] {Heat kernel on smooth metric measure spaces with nonnegative curvature}
\author{Jia-Yong Wu}
\address{Department of Mathematics, Shanghai Maritime University, Haigang Avenue 1550, Shanghai 201306, P. R. China}
\address{Department of Mathematics, Cornell University, Ithaca, NY 14853, United States} \email{jywu81@yahoo.com}

\

\author{Peng Wu}
\address{Department of Mathematics, Cornell University, Ithaca, NY 14853, United States}
\email{wupenguin@math.cornell.edu}

\thanks{}
\subjclass[2010]{Primary 35K08; Secondary 53C21, 58J35.}
\dedicatory{}
\date{\today}

\keywords{Bakry-\'{E}mery Ricci curvature, heat kernel, smooth metric measure space, Liouville theorem, $L_f^1$-uniqueness, gradient Ricci soliton, Moser iteration.}
\begin{abstract}
We derive a local Gaussian upper bound for the $f$-heat kernel on complete smooth metric measure space $(M,g,e^{-f}dv)$ with nonnegative Bakry-\'{E}mery Ricci curvature. As applications, we obtain a sharp $L_f^1$-Liouville theorem for $f$-subharmonic functions and an $L_f^1$-uniqueness property for nonnegative
solutions of the $f$-heat equation, assuming $f$ is of at most quadratic growth. In particular, any $L_f^1$-integrable $f$-subharmonic function on gradient shrinking and steady Ricci solitons must be constant. We also provide explicit $f$-heat kernel for Gaussian solitons.
\end{abstract}
\maketitle

\section{Introduction and main results}

In this paper we study Gaussian upper estimates for the $f$-heat kernel on smooth metric measure spaces with nonnegative Bakry-\'{E}mery Ricci curvature and their applications. Recall that a complete smooth metric measure space is a triple $(M,g,e^{-f}dv)$, where $(M,g)$ is an $n$-dimensional complete Riemannian manifold, $dv$ is the volume element of $g$, $f$ is a smooth function on $M$, and $e^{-f}dv$ (for short, $d\mu$) is called the weighted volume element or the weighted measure. The $m$-Bakry-\'{E}mery Ricci curvature \cite{[BE]} associated to $(M,g,e^{-f}dv)$ is defined by
\[
\mathrm{Ric}_f^m=\mathrm{Ric}+\nabla^2f-\frac{1}{m}df\otimes df,
\]
where $\mathrm{Ric}$ is the Ricci curvature of the manifold, $\nabla^2$ is the Hessian with respect to the metric $g$ and $m$ is a constant. We refer the readers to \cite{[BQ1]}, \cite{[LD]}, \cite{[Lott1]}, and \cite{[Lo-Vi]} for further details. When $m=\infty$, we write $\mathrm{Ric}_f=\mathrm{Ric}_f^{\infty}$. Smooth metric measure spaces are closely related to gradient Ricci solitons, the Ricci flow, probability theory, and optimal transport. A smooth metric measure space $(M,g, e^{-f}dv)$ is said to quasi-Einstein if
\[
\mathrm{Ric}_f^m=\lambda g
\]
for some constant $\lambda$. When $m=\infty$, it is exactly a gradient Ricci soliton. A gradient Ricci soliton is called expanding, steady or shrinking if $\lambda<0$, $\lambda=0$, and $\lambda>0$, respectively. Ricci solitons are natural extensions of Einstein manifolds and have drawn more and more attentions. See \cite{[Cao1]} for a nice survey and references therein.

The associated $f$-Laplacian $\Delta_f$ on a smooth metric measure space is defined as
\[
\Delta_f=\Delta-\nabla f\cdot\nabla,
\]
which is self-adjoint with respect to the weighted measure. On a smooth metric measure space, it is natural to consider the $f$-heat equation
\[
(\partial_t-\Delta_f)u=0
\]
instead of the heat equation. If $u$ is independent of time $t$, then $u$ is a $f$-harmonic function. Throughout this paper we denote by $H(x,y,t)$ the $f$-heat kernel, that is, for each $y\in M$, $H(x,y,t)=u(x,t)$ is the minimal positive solution of the $f$-heat equation with $\lim_{t\to 0}u(x,t)=\delta_{f,y}(x)$, where $\delta_{f,y}(x)$ is defined by
\[
\int_M\phi(x)\delta_{f,y}(x)e^{-f}dv=\phi(y)
\]
for $\phi\in C_0^{\infty}(M)$. Equivalently, $H(x,y,t)$ is the kernel of the semigroup $P_t=e^{t\Delta_f}$ associated to the Dirichlet energy $\int_M |\nabla \phi|^2e^{-f}dv$, where $\phi\in C_0^{\infty}(M)$. In general the $f$-heat kernel always exists on complete smooth metric measure spaces, but it may not be unique.
\vspace{0.5cm}

When $f$ is constant, then $H(x,y,t)$ is just the heat kernel for the Riemannian manifold $(M,g)$. Cheng, Li and Yau \cite{[ChLiYau]} obtained uniform Gaussian estimates for the heat kernel on Riemannian manifolds with sectional curvature bounded below, which was later extended by Cheeger, Gromov and Taylor \cite{[ChGrTa]} to manifolds with bounded geometry. In 1986, Li and Yau \cite{[Li-Yau1]} proved sharp Gaussian upper and lower bounds on Riemannian manifolds of nonnegative Ricci curvature, using the gradient estimate and the Harnack inequality. Grigor'yan and Saloff-Coste \cite{[Grig], [Saloff], [Saloff1], [Saloff2]}
independently proved similar estimates on Riemannian manifolds satisfying the volume doubling property and the Poincar\'e inequality, using the Moser iteration technique. Davies \cite{[Davies2]} further developed Gaussian upper bounds under a mean value property assumption. Recently, Li and Xu \cite{[LiXu]} also obtained some new estimates on complete Riemannian manifolds with Ricci curvature bounded from below by further improving the Li-Yau gradient estimate.

Recently, there have been several work on $f$-heat kernel estimates on smooth metric measure spaces and its applications. In \cite{[LD]}, X.-D. Li obtained Gaussian estimates for the $f$-heat kernel, and proved an $L_f^1$-Liouville theorem, assuming $\mathrm{Ric}^m_f$ ($m<\infty$) bounded below by a negative quadratic function, which generalizes a classical result of P. Li \cite{[Li0]}. He also mentioned that we may not be able to prove an $L_f^1$-Liouville theorem only assuming a lower bound on $\mathrm{Ric}_f$. The main difficulty is the lack of effective upper bound for the $f$-heat kernel. In \cite{[ChLu]}, by analyzing the heat kernel for a family of warped product manifolds, Charalambous and Lu also gave $f$-heat kernel estimates when $\mathrm{Ric}^m_f$ ($m<\infty$) is bounded below. In \cite{[Wu2]}, the first author proved $f$-heat kernel estimates assuming $Ric_f$ bounded below by a negative constant and $f$ bounded.

\

In this paper we prove a local Gaussian upper bound for the $f$-heat kernel on smooth metric measure spaces with $\mathrm{Ric}_f\geq 0$, which generalizes the classical result of Li-Yau \cite{[Li-Yau1]}.

\begin{theorem}\label{Main1}
Let $(M,g,e^{-f}dv)$ be an $n$-dimensional complete noncompact smooth metric measure space with $\mathrm{Ric}_f\geq 0$. Fix a fixed point $o\in M$ and $R>0$. For any $\epsilon>0$, there exist constants $c_1(n,\epsilon)$ and $c_2(n)$, such that
\begin{equation}\label{uppfu1}
H(x,y,t)\leq\frac{c_1(n,\epsilon)\,e^{c_2(n)A(R)}}{V_f(B_x(\sqrt{t}))^{1/2}V_f(B_y(\sqrt{t}))^{1/2}}
\times\exp\left(-\frac{d^2(x,y)}{(4+\epsilon)t}\right)
\end{equation}
for all $x,y\in B_o(\frac 12R)$ and $0<t<R^2/4$, where $\lim_{\epsilon\to 0}c_1(n,\epsilon)=\infty$. In particular, there exist constants $c_3(n,\epsilon)$ and $c_4(n)$, such that
\begin{equation}\label{lowhe2}
H(x,y,t)\leq\frac{c_3(n,\epsilon)\,e^{c_4(n)A(R)}}{V_f(B_x(\sqrt{t}))}\cdot\left(\frac{d(x,y)}{\sqrt{t}}+1\right)^{\frac n2}
\times\exp\left(-\frac{d^2(x,y)}{(4+\epsilon)t}\right)
\end{equation}
for any $x,y\in B_o(\frac 14R)$ and $0<t<R^2/4$, where $\lim_{\epsilon\to 0}c_3(n,\epsilon)=\infty$. Here $A(R):=\sup_{x\in B_o(3R)}|f|(x)$.
\end{theorem}

As pointed out by Munteanu-Wang \cite{[MuWa]}, only assuming $\mathrm{Ric}_f\geq 0$ may not be sufficient to derive $f$-heat kernel estimates by classical Li-Yau gradient estimate procedure \cite{[Li-Yau1]}. But we can derive a Gaussian upper bound using the De Giorgi-Nash-Moser theory and the weighted version of Davies's integral estimate \cite{[Davies]}.

\vspace{1em}

For Gaussian solitons, the $f$-heat kernel can be solved explicitly in closed forms.

\begin{example} $f$-heat kernel for steady Gaussian soliton.

Let $(\mathbb{R},\ g_0, e^{-f}dx)$ be a $1$-dimensional steady Gaussian soliton, where $g_0$ is the Euclidean metric and $f(x)=\pm x$. Then $\mathrm{Ric}_f=0$. The heat kernel of the operator $\Delta_f=\frac{d^2}{dx^2}\mp \frac{d}{dx}$ is given by
\[
H(x,y,t)=\frac{e^{\pm \frac{x+y}{2}}\cdot e^{-t/4}}{(4\pi t)^{1/2}}\times\exp\left(-\frac{|x-y|^2}{4t}\right).
\]
This $f$-heat kernel is solved using the separation of variables method, and it seems not in the literature, for the sake of completeness we include it in the appendix.
\end{example}

\begin{example}
Mehler heat kernel \cite{[Grig3]} for shrinking Gaussian soliton.

Let $(\mathbb{R},\ g_0, e^{-f}dx)$ be a $1$-dimensional shrinking Gaussian soliton, where $g_0$ is the Euclidean metric and $f(x)=x^2$. Then $\mathrm{Ric}_f=2$. The heat kernel of the operator $\Delta_f=\frac{d^2}{dx^2}-2x\frac{d}{dx}$ is given by
\[
H(x,y,t)=\frac{1}{(2\pi\sinh 2t)^{1/2}}\times\exp\left(\frac{2xye^{-2t}-(x^2+y^2)e^{-4t}}{1-e^{-4t}}+t\right).
\]
\end{example}

\begin{example}
Mehler heat kernel \cite{[Grig3]} for expanding Gaussian soliton.

Let $(\mathbb{R},\ g_0, e^{-f}dx)$ be a $1$-dimensional expanding Gaussian soliton, where $g_0$ is the Euclidean metric and $f(x)=-x^2$. Then $\mathrm{Ric}_f=-2$. The heat kernel of the operator $\Delta_f=\frac{d^2}{dx^2}+2x\frac{d}{dx}$ is given by
\[
H(x,y,t)=\frac{1}{(2\pi\sinh 2t)^{1/2}}\times\exp\left(\frac{2xye^{-2t}-(x^2+y^2)}{1-e^{-4t}}-t\right).
\]
\end{example}

\

As applications, we prove an $L_f^1$-Liouville theorem on complete smooth metric measure spaces with $\mathrm{Ric}_f\geq 0$ and $f$ to be of at most quadratic growth. We say $u\in L_f^p$, if $\int_M |u|^pe^{-f}dv<\infty$.

\begin{theorem}\label{Main2}
Let $(M,g,e^{-f}dv)$ be an $n$-dimensional complete noncompact smooth metric measure space with $\mathrm{Ric}_f\geq0$. Assume there exist nonnegative constants $a$ and $b$ such that
\[
|f|(x)\leq ar^2(x)+b\,\,\, {for}\,\, {all}\,\, x\in M,
\]
where $r(x)$ is the geodesic distance function to a fixed point $o\in M$. Then any nonnegative $L_f^1$-integrable $f$-subharmonic function must be identically constant. In particular, any $L_f^1$-integrable $f$-harmonic function must be identically constant.
\end{theorem}

From \cite{[CaoZhou]} and \cite{[Hami]} any complete noncompact shrinking or steady gradient Ricci soliton satisfies the assumptions in Theorem \ref{Main2}. Hence
\begin{corollary}\label{Shrinker}
Let $(M,g,e^{-f}dv)$ be a complete noncompact gradient shrinking or steady Ricci soliton. Then any nonnegative $L_f^1$-integrable $f$-subharmonic function must be identically constant.
\end{corollary}

\begin{remark}
Pigola, Rimoldi and Setti (see Corollary 23 in \cite{[PiRiSe]}) proved that on a complete gradient shrinking Ricci soliton, any locally lipschitz $f$-subharmonic function $u\in L_f^p$, $1<p<\infty$, is constant. Our result shows that this is true in the case $p=1$. Brighton \cite{[Bri]}, Cao-Zhou \cite{[CaoZhou]}, Munteanu-Sesum \cite{[MuSe]}, Munteanu-Wang \cite{[MuWa]}, Wei-Wylie \cite{[WW]} have proved several similar results.
\end{remark}

The growth condition of $f$ in Theorem \ref{Main2} is sharp as explained by the following simple example.
\begin{example}
Consider the $1$-dimensional smooth metric measure space $(\mathbb{R},\ g_0, e^{-f}dx)$, where $g_0$ is the Euclidean metric and $f(x)=x^{2+2\delta}$, $\delta=\frac{1}{2m+1}$ for $m\in\mathbb{N}$. By direct computation, $\mathrm{Ric}_f\geq0$. Let
\[
u(x):=\int_0^{|x|}e^{t^{2+2\delta}}dt.
\]
Then $u$ is $f$-harmonic. Moreover we claim $u\in L^1(\mu)$. Indeed, the integration by parts implies the identity
\[
\int_1^x e^{t^{2+2\delta}}dt=\frac{1}{2+2\delta}\left[
\frac{e^{x^{2+2\delta}}}{x^{1+2\delta}}-e+(1+2\delta)\int_1^x\frac{e^{t^{2+2\delta}}}{t^{2+2\delta}}dt\right].
\]
Then by L'Hospital rule, when $x$ is large enough,
\[
\int_1^x e^{t^{2+2\delta}}dt=\frac{1}{2+2\delta}\frac{e^{x^{2+2\delta}}}{x^{1+2\delta}}(1+o(1)).
\]
Therefore
\[
\int_{\mathbb{R}}ue^{-f} dx=\int_{-\infty}^{\infty}\Big(\int_0^{|x|} e^{t^{2+2\delta}}dt\Big)e^{-x^{2+2\delta}}dx <\infty,
\]
i.e., $u\in L_f^1$, but $u\not\in L_f^p$ for any $p>1$. On the other hand, if $\delta=0$ then $u\not\in L_f^1$.
\end{example}

By Theorem \ref{Main2}, we prove a uniqueness theorem for $L_f^1$-solutions of the $f$-heat equation, which generalizes the classical result of P. Li \cite{[Li0]}.
\begin{theorem}\label{Main3}
Let $(M,g,e^{-f}dv)$ be an $n$-dimensional complete noncompact smooth metric measure space with $\mathrm{Ric}_f\geq0$. Assume there exist nonnegative constants $a$ and $b$ such that
\[
|f|(x)\leq ar^2(x)+b\,\,\, {for}\,\, {all}\,\, x\in M,
\]
where $r(x)$ is the distance function to a fixed point $o\in M$. If $u(x,t)$ is a nonnegative function defined on $M\times[0,+\infty)$ satisfying
\[
(\partial_t-\Delta_f)u(x,t)\leq0, \quad\int_Mu(x,t)e^{-f}dv<+\infty
\]
for all $t>0$, and
\[
\lim_{t\to 0}\int_Mu(x,t)e^{-f}dv=0,
\]
then $u(x,t)\equiv0$ for all $x\in M$ and $t\in(0,+\infty)$. In particular, any $L_f^1$-solution of the $f$-heat equation is uniquely determined by its initial data in $L_f^1$.
\end{theorem}

The rest of the paper is organized as follows. In Section \ref{sec2a}, we provide a relative $f$-volume comparison theorem for smooth metric measure spaces with nonnegative Bakry-\'{E}mery Ricci curvature. Using the comparison theorem, we derive a local $f$-volume doubling property, a local $f$-Neumann Poincar\'e inequality, a local Sobolev inequality, and a $f$-mean value inequality. In Section \ref{sec3}, we prove local Gaussian upper bounds of the $f$-heat kernel by applying the mean value inequality. In Sections \ref{sec4} and \ref{sec5}, we prove the $L_f^1$-Liouville theorem for $f$-subharmonic functions and the $L_f^1$-uniqueness property for nonnegative solutions of the $f$-heat equation following the argument of Li in \cite{[Li0]}. In appendix, we compute the $f$-heat kernel of $1$-dimensional steady Gaussian soliton.

\vspace{0.5em}

\textbf{Acknowledgements}. The authors thank Professors Xiaodong Cao and Zhiqin Lu for helpful discussions. The second author thanks Professors Xianzhe Dai and Guofang Wei for helpful discussions, constant encouragement and support. The first author is partially supported by NSFC (11101267, 11271132) and the China Scholarship Council (08310431). The second author is partially supported by an AMS-Simons travel grant.


\section{Poincar\'e, Sobolev and mean value inequalities}
\label{sec2a}
Let $\Delta_f=\Delta-\nabla f\cdot\nabla$ be the $f$-Laplacian on a complete smooth metric measure space $(M,g,e^{-f}dv)$. Throughout this section, we will assume
\[
\mathrm{Ric}_f\geq 0.
\]
For a fixed point $o\in M$ and $R>0$, we denote
\[
A(R)=\sup_{x\in B_o(3R)}|f(x)|.
\]
We often write $A$ for short. First we have the relative $f$-volume comparison theorem proved by Wei and Wylie \cite{[WW]} and Munteanu-Wang \cite{[MuWa]}.

\begin{lemma}\label{comp}
Let $(M,g,e^{-f}dv)$ be an $n$-dimensional complete noncompact smooth metric measure space. If $\mathrm{Ric}_f\geq0$, then for any $x\in B_o(R)$,
\begin{equation}\label{volcomp}
\frac{V_f(B_x(R_1,R_2))}{V_f(B_x(r_1,r_2))}\leq e^{4A}\frac{R^n_{2}-R^n_{1}}{r^n_{2}-r^n_{1}},
\end{equation}
for any $0<r_1<r_2,\ 0<R_1<R_2<R$, $r_1\leq R_1,\ r_2\leq R_2<R$, where $B_x(R_1,R_2):=B_x(R_2)\backslash B_x(R_1)$.
\end{lemma}

\begin{proof}[Proof of Lemma \ref{comp}]
Wei and Wylie (see (3.19) in \cite{[WW]}) proved the following $f$-mean curvature comparison theorem. Recall that the weighted mean curvature $m_f(r)$ is defined as $m_f(r)=m(r)-\nabla f\cdot\nabla r=\Delta_f\ r$. For any $x\in B_o(R)\subset M$, if $\mathrm{Ric}_f\geq 0$, then
\[
m_f(r)\leq\frac{n-1}{r}+\frac{2}{r^2}\int^r_0f(t)dt-\frac{2}{r}f(r),
\]
along any minimal geodesic segment from $x$.

In geodesic polar coordinates, the volume element is written as $dv=\mathcal{A}(r,\theta)dr\wedge d\theta_{n-1}$, where $d\theta_{n-1}$ is the standard volume element of the unit sphere $S^{n-1}$. Let $\mathcal{A}_f(r,\theta)=e^{-f}\mathcal{A}(r,\theta)$. By the first variation of the area,
\[
\frac{\mathcal{A'}}{\mathcal{A}}(r,\theta)=(\ln(\mathcal{A}(r,\theta)))'=m(r,\theta).
\]
Therefore
\[
\frac{\mathcal{A'}_f}{\mathcal{A}_f}(r,\theta)=(\ln(\mathcal{A}_f(r,\theta)))'=m_f(r,\theta).
\]
For $0<r_1<r_2<R$, integrating this from $r_1$ to $r_2$ we get
\begin{equation*}
\begin{aligned}
\frac{\mathcal{A}_f(r_2,\theta)}{\mathcal{A}_f(r_1,\theta)}&=\exp\left(\int_{r_1}^{r_2}m_f(s,\theta)ds\right)\\
&\leq\left(\frac{r_2}{r_1}\right)^{n-1}\exp\left[2\int_{r_1}^{r_2}\frac{1}{r^2}\left(\int^r_0f(t)dt\right)dr
-2\int_{r_1}^{r_2}\frac{f(r)}{r}dr\right].
\end{aligned}
\end{equation*}
Since
\[
\int_{r_1}^{r_2}\frac{1}{r^2}\left(\int^r_0f(t)dt\right)dr
=-\frac{1}{r}\left(\int^r_0f(t)dt\right)\Big|_{r_1}^{r_2}+\int_{r_1}^{r_2}\frac{f(r)}{r}dr,
\]
then we have
\[
\frac{\mathcal{A}_f(r_2,\theta)\cdot\exp(\frac{2}{r_2}\int^{r_2}_0f(t)dt)}{\mathcal{A}_f(r_1,\theta)
\cdot\exp(\frac{2}{r_1}\int^{r_1}_0f(t)dt)}\leq\left(\frac{r_2}{r_1}\right)^{n-1}
\]
for $0<r_1<r_2<R$. That is
$
r^{1-n}\mathcal{A}_f(r,\theta)\exp(\frac{2}{r}\int^{r}_0f(t)dt)
$
is nonincreasing in $r$. Applying Lemma 3.2 in \cite{[Zhu]}, we get
\[
\frac{\int^{R_2}_{R_1}\mathcal{A}_f(r,\theta)\cdot\exp(\frac{2}{t}\int^{t}_0f(\tau)d\tau)dt}
{\int^{r_2}_{r_1}\mathcal{A}_f(r,\theta)\cdot\exp(\frac{2}{t}\int^{t}_0f(\tau)d\tau)dt}
\leq\frac{\int^{R_2}_{R_1}t^{n-1}dt}{\int^{r_2}_{r_1}t^{n-1}dt}
\]
for $0<r_1<r_2$, $0<R_1<R_2$, $r_1\leq R_1$ and $r_2\leq R_2<R$. Integrating along the sphere direction gives
\[
\frac{V_f(B_x(R_1,R_2))}{V_f(B_x(r_1,r_2))}\leq
e^{4A}\frac{R^n_{2}-R^n_{1}}{r^n_{2}-r^n_{1}},
\]
for any $0<r_1<r_2,\ 0<R_1<R_2<R$, $r_1\leq R_1,\ r_2\leq R_2<R$, where $B_x(R_1,R_2)=B_x(R_2)\backslash B_x(R_1)$. 
\end{proof}

From \eqref{volcomp}, letting $r_1=R_1=0$, $r_2=r$ and $R_2=2r$, we get
\begin{equation}\label{voldop}
V_f(B_x(2r))\leq 2^ne^{4A}\cdot V_f(B_x(r))
\end{equation}
for any $0<r<R/2$. This inequality implies that the local $f$-volume doubling property holds. This property will play a crucial role in our paper. We say that a complete smooth metric measure space $(M,g,e^{-f}dv)$ satisfies the local $f$-volume doubling property if for any $0<R<\infty$, there exists a constant $C(R)$ such that
\[
V_f(B_x(2r))\leq C(R)\cdot V_f(B_x(r))
\]
for any $0<r<R$ and $x\in M$. Note that when the above inequality holds with $R=\infty$, then it is called the global $f$-volume doubling property.\\

From Lemma \ref{comp}, we have

\begin{lemma}\label{lecomp1}
Let $(M,g,e^{-f}dv)$ be an $n$-dimensional complete noncompact smooth metric measure space. If $\mathrm{Ric}_f\geq 0$, then
\[
\frac{V_f(B_x(s))}{V_f(B_y(r))}\leq  4^ne^{8A}\left(\frac sr\right)^\kappa,
\]
where $\kappa=\log_2(2^ne^{4A})$, for any $0<r<s<R/4$ and all $x\in B_o(s)$ and $y\in B_x(s)$. Moreover, we have
\[
V_f(B_x(r))\leq e^{4A}\left(\frac{d(x,y)}{r}+1\right)^n V_f(B_y(r))
\]
for any $x,y\in B_o(\frac 14R)$ and $0<r<R/2$.
\end{lemma}

\begin{proof}
Choose a real number $k$ such that $2^k<s/r\leq 2^{k+1}$. Since $y\in B_x(s)$,
\[
B_x(s)\subset B_y(2s)\subset B_y(2^{k+2}r),
\]
and $V_f(B_x(s))\leq V_f(B_y(2^{k+2}r))$. Moreover, the assumption $\mathrm{Ric}_f\geq 0$ implies the local $f$-volume doubling property \eqref{voldop}. So we have
\[
V_f(B_x(s))\leq(2^ne^{4A})^{k+2} V_f(B_y(r))\leq (2^ne^{4A})^2(s/r)^\kappa V_f(B_y(r)),
\]
where $\kappa=\log_2(2^ne^{4A})$. This proves the first part of the lemma.

For the second part, letting $r_1=0$, $r_2=r$, $R_1=d(x,y)-r$ and $R_2=d(x,y)+r$ in Lemma \ref{comp}, we have
\[
\frac{V_f(B_x(d(x,y)+r))-V_f(B_x(d(x,y)-r))}{V_f(B_x(r))}\leq e^{4A}\left(\frac{d(x,y)}{r}+1\right)^n
\]
for any $x,y\in B_o(\frac 14R)$ and $0<r<R/2$. Therefore we get
\begin{equation*}
\begin{aligned}
V_f(B_x(r))&\leq V_f(B_y(d(x,y)+r))-V_f(B_y(d(x,y)-r))\\
&\leq e^{4A}\left(\frac{d(x,y)}{r}+1\right)^n V_f(B_y(r))
\end{aligned}
\end{equation*}
for any $x,y\in B_o(\frac 14R)$ and $0<r<R/2$. 
\end{proof}

By Lemma \ref{comp}, following Buser's proof \cite{[Bus]} or Saloff-Coste's alternate proof (Theorem 5.6.5  in \cite{[Saloff2]}), we get a local Neumann Poincar\'e inequality on smooth metric measure spaces, see also Munteanu and Wang (see Lemma 3.1 in \cite{[MuWa]}).

\begin{lemma}\label{NeuPoin}
Let $(M,g,e^{-f}dv)$ be an $n$-dimensional complete noncompact smooth metric measure space with $\mathrm{Ric}_f\geq 0$. 
Then for any $x\in B_o(R)$,
\begin{equation}\label{Nepoinineq}
\int_{B_x(r)}|\varphi-\varphi_{B_x(r)}|^2e^{-f}dv\leq c_1e^{c_2A}\cdot r^2\int_{B_x(r)}|\nabla \varphi|^2e^{-f}dv
\end{equation}
for all $0<r<R$ and $\varphi\in C^\infty(B_x(r))$, where $\varphi_{B_x(r)}=V_f^{-1}(B_x(r))\int_{B_x(r)}\varphi e^{-f}dv$.
The constants $c_1$ and $c_2$ depend only on $n$.
\end{lemma}

\begin{remark}
When $f$ is constant, this was classical result of Saloff-Coste (see (6) in \cite{[Saloff1]} or Theorem 5.6.5 in \cite{[Saloff2]}).
\end{remark}

Combining Lemma \ref{comp}, Lemma \ref{lecomp1}, Lemma \ref{NeuPoin} and the argument in \cite{[Saloff]}, we obtain a local Sobolev inequality.

\begin{lemma}\label{NeuSob}
Let $(M,g,e^{-f}dv)$ be an $n$-dimensional complete noncompact smooth metric measure space with $\mathrm{Ric}_f\geq0$. Then there exist constants $p>2$, $c_3$ and $c_4$, all depending only on $n$ such that
\begin{equation}\label{loSob}
\left(\int_{B_o(r)}|\varphi|^{\frac{2p}{p-2}}e^{-f}dv\right)^{\frac{p-2}{p}}\leq
\frac{c_3e^{c_4A}\cdot r^2}{V_f(B_o(r))^{\frac 2p}}\int_{B_o(r)}(|\nabla \varphi|^2+r^{-2}|\varphi|^2)e^{-f}dv
\end{equation}
for any $x\in M$ such that $0<r(x)<R$ and $\varphi\in C^\infty(B_o(r))$.
\end{lemma}

\begin{proof}[Sketch proof of Lemma \ref{NeuSob}]
The proof is essentially a weighted version of Theorem 2.1 in \cite{[Saloff]} (see also Theorem 3.1 in \cite{[Saloff1]}).

Besides, we have an alternate proof by applying the local Neumann Sobolev inequality of Munteanu and Wang (see Lemma 3.2 in \cite{[MuWa]})
\[
\|\varphi-\varphi_{B_o(r)}\|_{\frac{2p}{p-2}}\leq \frac{c_3e^{c_4A}\cdot r}{V_f(B_o(r))^{\frac 1p}}\|\nabla\varphi\|_2,
\]
where $\|f\|_m=(\int_{B_o(r)}|f|^md\mu)^{1/m}$. Munteanu and Wang proved this inequality holds without the weighted measure, and it is still true by checking their proof when integrals are with respect to the weighted volume element $e^{-f}dv$. Combining this with the Minkowski inequality
\[
\|\varphi\|_{\frac{2p}{p-2}}\leq\|\varphi-\varphi_{B_o(r)}\|_{\frac{2p}{p-2}}+\|\varphi_{B_o(r)}\|_{\frac{2p}{p-2}},
\]
it is sufficient to prove
\begin{equation*}\label{ineqinte}
\|\varphi_{B_o(r)}\|_{\frac{2p}{p-2}}\leq\frac{c_3e^{c_4A}}{V_f(B_o(r))^{\frac 1p}}\|\varphi\|_2,
\end{equation*}
which follows from Cauchy-Schwarz inequality. Hence the lemma follows.
\end{proof}

Lemma \ref{NeuSob} is a critical step in proving the Harnack inequality by the Moser iteration technique \cite{[Moser]}. We apply it to prove a local mean value inequality for the $f$-heat equation, which is similar to the case when $f$ is constant, obtained by Saloff-Coste \cite{[Saloff]} and Grigor'yan \cite{[Grig]}.

\begin{proposition}\label{PoinDouHarn}
Let $(M,g,e^{-f}dv)$ be an $n$-dimensional complete noncompact smooth metric measure space. Fix $R>0$. Assume that \eqref{loSob} holds up to $R$. Then there exist constants $c_5(n,p)$ and $c_6(n,p)$ such that, for any real $s$, for any $0<\delta<\delta'\leq1$, and for any smooth positive solution $u$ of the $f$-heat equation in the cylinder $Q=B_o(r)\times(s-r^2,s)$, $r<R$, we have
\begin{equation}\label{L1-mean}
\sup_{Q_\delta}\{u\}\leq \frac{c_5e^{c_6A}}{(\delta'-\delta)^{2+p}\,r^2\,V_f(B_o(r))}\cdot\int_{Q_{\delta'}}u\,\,\, d\mu\, dt,
\end{equation}
where $Q_\delta=B_o(\delta r)\times (s-\delta r^2,s)$ and $Q_{\delta'}=B_o(\delta' r)\times (s-\delta' r^2,s)$.
\end{proposition}

\begin{proof}
The proof is analogous to Theorem 5.2.9 in \cite{[Saloff2]}. For the readers convenience, we provide a detailed proof. We need to carefully examine the explicit coefficients of the mean value inequality in terms of the Sobolev constants in \eqref{loSob}.

Without loss of generality we assume $\delta'=1$. For any nonnegative function $\phi\in C^{\infty}_0(B)$, $B=B_o(r)$, we have
\[
\int_B(\phi u_t+\nabla\phi\nabla u)d\mu=0.
\]
Let $\phi=\psi^2u$, $\psi\in C^{\infty}_0(B)$, then
\begin{equation*}
\begin{aligned}
\int_B(\psi^2 uu_t+\psi^2|\nabla u|^2)d\mu&\leq 2\left|\int_B u\psi \nabla u\nabla\psi d\mu\right|\\
&\leq 3\int_B|\nabla\psi|^2u^2d\mu+\frac 13\int_B\psi^2|\nabla u|^2d\mu,
\end{aligned}
\end{equation*}
so we get that
\[
\int_B(2\psi^2 uu_t+|\nabla(\psi u)|^2)d\mu\leq 10\|\nabla\psi\|^2_{\infty} \int_{\mathrm{supp}(\psi)} u^2d\mu.
\]
Multiplying a smooth function $\lambda(t)$, which will be determined later, from the above inequality, we get
\begin{equation*}
\begin{aligned}
\frac{\partial}{\partial t}\left(\int_B(\lambda\psi u)^2d\mu\right)+\lambda^2\int_B|\nabla(\psi u)|^2d\mu
\leq C\lambda\left(\lambda\|\nabla \psi\|^2_{\infty}+|\lambda'|\sup\psi^2\right)\int_{\mathrm{supp}(\psi)} u^2d\mu,
\end{aligned}
\end{equation*}
where $C$ is a finite constant which will change from line to line in the following inequalities.

Next we choose $\psi$ and $\lambda$ such that, for any $0<\sigma'<\sigma<1$, $\kappa=\sigma-\sigma'$,
\begin{enumerate}
\item
$0\leq\psi\leq 1$, $\mathrm{supp}(\psi)\subset\sigma B$, $\psi=1$ in $\sigma' B$ and $|\nabla\psi|\leq 2(\kappa r)^{-1}$;
\item
$0\leq\lambda\leq 1$, $\lambda=0$ in $(-\infty,s-\sigma r^2)$,  $\lambda=1$ in $(s-\sigma' r^2,+\infty)$, and $|\lambda'(t)|\leq 2(\kappa r)^{-2}$.
\end{enumerate}
Let $I_\sigma=(s-\sigma r^2,s)$ and $I_\sigma'=(s-\sigma' r^2,s)$. For any $t\in I_{\sigma'}$, integrating the above inequality over $(s-r^2,t)$,
\begin{equation}
\begin{aligned}\label{integso}
\sup_{I_{\sigma'}}\left\{\int_B\psi u^2d\mu\right\}+\int_{B\times I_{\sigma'}}|\nabla(\psi u)|^2d\mu dt
\leq C\left(r \kappa\right)^{-2}\int_{Q_\sigma} u^2d\mu dt.
\end{aligned}
\end{equation}
On the other hand, by the H\"older inequality and the assumption of proposition, for some $p>2$, we have
\begin{equation}
\begin{aligned}\label{ints}
\int_B\varphi^{2(1+\frac 2p)}d\mu
&\leq \left(\int_B|\varphi|^{\frac{2p}{p-2}}d\mu\right)^{\frac{p-2}{p}}\cdot\left(\int_B\varphi^2d\mu\right)^{\frac{2}{p}}\\
&\leq \left(\int_B\varphi^2d\mu\right)^{\frac{2}{p}}\cdot\left(E(B)\int_B(|\nabla \varphi|^2+r^{-2}|\varphi|^2)d\mu\right)
\end{aligned}
\end{equation}
for all $\varphi\in C^{\infty}_0(B)$, where $E(B)=c_3e^{c_4A}r^2V_f(B_o(r))^{-2/p}$. Combining \eqref{integso} and \eqref{ints}, we get
\[
\int_{Q_{\sigma'}}u^{2\theta}d\mu dt
\leq E(B)\left[C(r\kappa)^{-2}\int_{Q_\sigma} u^2d\mu dt\right]^\theta
\]
with $\theta=1+2/p$. For any $m\geq 1$, $u^m$ is also a smooth positive solution of $(\partial_t-\Delta_f)u(x,t)\leq0$. Hence the above inequality indeed implies
\begin{equation}\label{intds2}
\int_{Q_{\sigma'}}u^{2m\theta}d\mu dt
\leq E(B)\left[C(r\kappa)^{-2}\int_{Q_\sigma} u^{2m}d\mu dt\right]^\theta
\end{equation}
for $m\geq1$.

Let $\kappa_i=(1-\delta)2^{-i}$, which satisfies $\Sigma^{\infty}_1\kappa_i=1-\delta$. Let $\sigma_0=1$, $\sigma_{i+1}=\sigma_i-\kappa_i=1-\Sigma^i_1\kappa_j$. Applying \eqref{intds2} for $m=\theta^i$, $\sigma=\sigma_i$, $\sigma'=\sigma_{i+1}$, we have
\[
\int_{Q_{\sigma_{i+1}}}u^{2\theta^{i+1}}d\mu dt
\leq E(B)\left[C^{i+1}((1-\delta)r)^{-2}\int_{Q_{\sigma_i}} u^{2\theta^i}d\mu dt\right]^\theta.
\]
Therefore
\[
\left(\int_{Q_{\sigma_{i+1}}}u^{2\theta^{i+1}}d\mu dt\right)^{\theta^{-(i+1)}}
\leq  C^{\Sigma j\theta^{1-j}}\cdot E(B)^{\Sigma\theta^{-j}}\cdot[(1-\delta)r]^{-2\Sigma\theta^{1-j}}\int_Q u^2d\mu dt,
\]
where $\Sigma$ denotes the summations from $1$ to $i+1$. Letting $i\to \infty$ we get
\begin{equation}\label{prmi}
\sup_{Q_\delta}\{u^2\}\leq C\cdot E(B)^{p/2}\cdot[(1-\delta)r]^{-2-p}||u||^2_{2,Q}
\end{equation}
for some $p>2$.

Formula \eqref{prmi} in fact is a $L_f^2$-mean value inequality. Next, we will apply \eqref{prmi} to prove \eqref{L1-mean} by a different iterative argument. Let $\sigma\in (0,1)$ and $\rho=\sigma+(1-\sigma)/4$. Then \eqref{prmi} implies
\[
\sup_{Q_\sigma}\{u\}\leq F(B)\cdot(1-\sigma)^{-1-p/2}||u||_{2,Q_{\rho}},
\]
where $F(B)=c_3e^{c_4A}\cdot r^{-1}\cdot V_f(B_o(r))^{-1/2}$. Since
\[
\|u\|_{2,Q}\leq \|u\|^{1/2}_{\infty,Q}\cdot\|u\|^{1/2}_{1,Q}
\]
for any $Q$, so we have
\begin{equation}\label{moserit}
\|u\|_{\infty,Q_\sigma} \leq F(B)\cdot\|u\|^{1/2}_{1,Q}\cdot(1-\sigma)^{-1-p/2}\|u\|^{1/2}_{\infty,Q_{\rho}}.
\end{equation}

Now fix $\delta\in (0,1)$ and let $\sigma_0=\delta$, $\sigma_{i+1}=\sigma_i+(1-\sigma_i)/4$, which satisfy $1-\sigma_i=(3/4)^i(1-\delta)$. Applying \eqref{moserit} to $\sigma=\sigma_i$ and $\rho=\sigma_{i+1}$, we have
\[
\|u\|_{\infty,Q_{\sigma_i}} \leq (4/3)^{(1+p/2)i}F(B)\cdot\|u\|^{1/2}_{1,Q}\cdot(1-\delta)^{-1-p/2}\|u\|^{1/2}_{\infty,Q_{\sigma_{i+1}}}.
\]
Therefore, for any $i$,
\[
\|u\|_{\infty,Q_{\delta}} \leq (4/3)^{(1+p/2)\Sigma j(\frac 12)^j}\times\left[F(B)\cdot\|u\|^{1/2}_{1,Q}\cdot(1-\delta)^{-1-p/2}\right]^{\Sigma(\frac 12)^j}
\|u\|^{(\frac 12)^i}_{\infty,Q_{\sigma_i}},
\]
where $\Sigma$ denotes the summations from $0$ to $i-1$. Letting $i\to \infty$ we get
\[
\|u\|_{\infty,Q_{\delta}}\leq (4/3)^{(2+p)}\left[F(B)\cdot\|u\|^{1/2}_{1,Q}\cdot(1-\delta)^{-1-p/2}\right]^2,
\]
that is,
\[
\|u\|_{\infty,Q_{\delta}}\leq (4/3)^{(2+p)}c_5e^{c_6A}(1-\delta)^{-2-p}\cdot r^{-2}\cdot V_f(B_o(r))^{-1}\cdot\|u\|_{1,Q}
\]
and the proposition follows.
\end{proof}


\section{Gaussian upper bounds of the $f$-heat kernel}\label{sec3}

In this section, we prove Gaussian upper bounds of the $f$-heat kernel on smooth metric measure spaces with nonnegative Bakry-\'{E}mery Ricci curvature by applying Proposition \ref{PoinDouHarn} and Lemma \ref{lecomp1}. To prove Theorem \ref{Main1}, first we need a weighted version of Davies' integral estimate \cite{[Davies]}.

\begin{lemma}\label{lemm3.3}
Let $(M,g,e^{-f}dv)$ be an $n$-dimensional complete smooth metric measure space. Let $\lambda_1(M)\geq0$ be the bottom of the $L_f^2$-spectrum of the $f$-Laplacian on $M$. Assume that $B_1$ and $B_2$ are bounded subsets of $M$. Then
\begin{equation}\label{Dvpp}
\int_{B_1}\int_{B_2}H(x,y,t)d\mu(x)d\mu(y)\leq
V_f(B_1)^{1/2}V_f(B_2)^{1/2}\exp\left(-\lambda_1(M)t-\frac{d^2(B_1,B_2)}{4t}\right),
\end{equation}
where $d(B_1,B_2)$ denotes the distance between the sets $B_1$ and $B_2$.
\end{lemma}

\begin{proof}[Proof of Lemma \ref{lemm3.3}]
By the approximation argument, it suffices to prove \eqref{Dvpp} for the $f$-heat kernel $H_\Omega$ of any compact set with boundary $\Omega$ containing $B_1$ and $B_2$. In fact, let $\Omega_i$ be a sequence of compact exhaustion of $M$ such that $\Omega_i\subset\Omega_{i+1}$ and $\cup_{i}\Omega_i=M$. If we prove \eqref{Dvpp} for the $f$-heat kernel $H_{\Omega_i}$ for any $i$, then the lemma follows by letting $i\to\infty$ and observing that $\lambda_1(\Omega_i)\to\lambda_1(M)$, where $\lambda_1(\Omega_i)>0$ is the first Dirichlet eigenvalue of the $f$-Laplacian on $\Omega_i$, and
$\lambda_1(M)=\inf_{\Omega_i\subset M}\lambda_1(\Omega_i)$.

We consider the function $u(x,t)=e^{t\Delta_f|_\Omega}\textbf{1}_{B_1}$ with Dirichlet boundary condition: $u(x,t)=0$ on $\partial \Omega$. Then
\begin{equation}\label{Dvpp1}
\begin{aligned}
\int_{B_2}\int_{B_1}H_\Omega(x,y,t)d\mu(y)d\mu(x)
&=\int_{B_2}\left(\int_{\Omega}H_\Omega(x,y,t)\textbf{1}_{B_1}d\mu(y)\right)d\mu(x)\\
&=\int_{B_2}u(x,t)d\mu(x)\\
&\leq V_f(B_2)^{1/2}\cdot\left(\int_{B_2}u^2(x,t)d\mu(x)\right)^{1/2}.
\end{aligned}
\end{equation}
For some $\alpha>0$, we define $\xi(x,t)=\alpha d(x,B_1)-\frac{\alpha^2}{2}t$ and consider the function
\[
J(t):=\int_\Omega u^2(x,t)e^{\xi(x,t)}d\mu(x).
\]

\textbf{Claim}:  Function $J(t)$ satisfies
\begin{equation}\label{Jmonoto}
J(t)\leq J(t_0)\cdot\exp(-2\lambda_1(\Omega)(t-t_0))
\end{equation}
for all $0<t_0\leq t$.

This claim will be proved later. We now continue to prove Lemma \ref{lemm3.3}. If $x\in B_2$, then $\xi(x,t)\geq\alpha d(B_2,B_1)-\frac{\alpha^2}{2}t$. Hence
\begin{equation}\label{Dvpp2}
\begin{aligned}
J(t)&\geq\int_{B_2} u^2(x,t)e^{\xi(x,t)}d\mu(x)\\
&\geq\exp\left(\alpha d(B_2,B_1)-\frac{\alpha^2}{2}t\right)\int_{B_2} u^2(x,t)d\mu(x).
\end{aligned}
\end{equation}
On the other hand, if $x\in B_1$ then $\xi(x,0)=0$. Using \eqref{Jmonoto} and the continuity of $J(t)$ at $t=0^+$, we have
\begin{equation}\label{Dvpp3}
\begin{aligned}
J(t)&\leq J(0)\cdot\exp\left(-2\lambda_1(\Omega)t\right)\\
&=\int_\Omega e^{\xi(x,0)}\textbf{1}_{B_1}d\mu(x)\cdot\exp\left(-2\lambda_1(\Omega)t\right)\\
&=V_f(B_1)\cdot\exp\left(-2\lambda_1(\Omega)t\right)
\end{aligned}
\end{equation}
Combining \eqref{Dvpp1}, \eqref{Dvpp2} and \eqref{Dvpp3}, and choosing $\alpha=d(B_1,B_2)/t$, we get
\[
\int_{B_1}\int_{B_2}H_\Omega(x,y,t)d\mu(x)d\mu(y)\leq
V_f(B_1)^{1/2}V_f(B_2)^{1/2}\exp\left(-\lambda_1(\Omega)t-\frac{d^2(B_1,B_2)}{4t}\right)
\]
for any compact set $\Omega\subset M$. Lemma \ref{lemm3.3} is proved.

\qed

\textbf{Proof of the claim}. Since $\xi_t\leq -\frac 12|\nabla \xi|^2$ and $u_t=\Delta_fu$, we compute directly
\begin{equation}\label{diff}
\begin{aligned}
J'(t)&\leq2\int_\Omega u \Delta_fue^{\xi}d\mu(x)-\frac 12\int_\Omega u^2e^{\xi}|\nabla \xi|^2d\mu(x)\\
&=-2\int_\Omega |\nabla u|^2e^{\xi}d\mu(x)-2\int_\Omega u\langle\nabla u,\nabla \xi\rangle e^{\xi}d\mu(x)
-\frac 12\int_\Omega u^2e^{\xi}|\nabla \xi|^2d\mu(x)\\
&=-2\int_\Omega (u\nabla\xi+2\nabla u)^2e^{\xi}d\mu(x)\\
&=-2\int_\Omega |\nabla(u e^{\xi/2})|^2d\mu(x).
\end{aligned}
\end{equation}
Moreover the definition of $\lambda_1(\Omega)$ implies
\[
\int_\Omega |\nabla(u e^{\xi/2})|^2d\mu(x)\geq\lambda_1(\Omega)\int_\Omega |u e^{\xi/2}|^2d\mu(x)
=\lambda_1(\Omega)J(t).
\]
Substituting this into \eqref{diff} we get $J'(t)\leq-2\lambda_1(\Omega)J(t)$ and the claim is proved.
\end{proof}
\

Now we prove the upper bounds of $f$-heat kernel by modifying the argument of \cite{[Davies2]} (see also \cite{[Li2]}).

\begin{proof}[Proof of Theorem \ref{Main1}]
We denote $u:(y,s)\mapsto H(x,y,s)$ be a $f$-heat kernel. Under the assumption $t\geq r^2_2$, applying $u$ to Proposition \ref{PoinDouHarn} with a fixed $x\in B_o(R/2)$, we have
\begin{equation}\label{itegA}
\begin{aligned}
\sup_{(y,s)\in Q_\delta}H(x,y,s)&\leq \frac{c_5e^{c_6A}}{r^2_2V_f(B_2)}\cdot\int^t_{t-1/4r^2_2}\int_{B_2}H(x,\zeta,s) d\mu(\zeta)ds\\
&=\frac{c_5e^{c_6A}}{4V_f(B_2)}\cdot\int_{B_2}H(x,\zeta,s') d\mu(\zeta)
\end{aligned}
\end{equation}
for some $s'\in(t-1/4r^2_2, t)$, where $Q_\delta=B_y(\delta r_2)\times(t-\delta r^2_2, t)$ with $0<\delta<1/4$, and $B_2=B_y(r_2)\subset B_o(R)$ for $y\in B_o(R/2)$. Applying Proposition \ref{PoinDouHarn} and the same argument to the positive solution
\[
v(x,s)=\int_{B_2}H(x,\zeta,s) d\mu(\zeta)
\]
of the $f$-heat equation, for the variable $x$ with $t\geq r^2_1$, we also have
\begin{equation}\label{itegB}
\begin{aligned}
\sup_{(x,s)\in \overline{Q}_\delta}\int_{B_2}H(x,\zeta,s) d\mu(\zeta)&\leq\frac{c_5e^{c_6A}}{r^2_1V_f(B_1)}\cdot\int^t_{t-1/4r^2_1}
\int_{B_1}\int_{B_2}H(\xi,\zeta,s)d\mu(\zeta)d\mu(\xi)ds\\
&=\frac{c_5e^{c_6A}}{4V_f(B_1)}\cdot\int_{B_1}\int_{B_2}H(\xi,\zeta,s'')d\mu(\zeta)d\mu(\xi)
\end{aligned}
\end{equation}
for some $s''\in(t-1/4r^2_1, t)$, where $\overline{Q}_\delta=B_x(\delta r_1)\times(t-\delta r^2_1, t)$ with $0<\delta<1/4$, and $B_1=B_x(r_1)\subset B_o(R)$ for $x\in B_o(R/2)$. Now letting $r_1=r_2=\sqrt{t}$ and combining \eqref{itegA} with \eqref{itegB}, the $f$-heat kernel satisfies
\begin{equation}\label{heaup1}
H(x,y,t)\leq\frac{(c_5e^{c_6A})^2}{V_f(B_1)V_f(B_2)}\cdot\int_{B_1}\int_{B_2}H(\xi,\zeta,s'')d\mu(\zeta)d\mu(\xi)
\end{equation}
for all $x,y\in B_o(R/2)$ and $0<t<R^2/4$. Using Lemma \ref{lemm3.3} and noticing that $s''\in(\frac 34t, t)$, \eqref{heaup1} becomes
\begin{equation}\label{heaup2}
H(x,y,t)\leq\frac{(c_5e^{c_6A})^2}{V_f(B_x(\sqrt {t}))^{1/2}V_f(B_y(\sqrt {t}))^{1/2}}
\times\exp\left(-\frac 34\lambda_1t-\frac{d^2(B_1,B_2)}{4t}\right)
\end{equation}
for all $x,y\in B_o(R/2)$ and $0<t<R^2/4$. Notice that if $d(x,y)\leq 2\sqrt{t}$, then $d(B_x(\sqrt {t}),B_y(\sqrt {t}))=0$ and hence
\[
-\frac{d^2(B_x(\sqrt{t}),B_y(\sqrt{t}))}{4t}=0\leq 1-\frac{d^2(x,y)}{4t},
\]
and if $d(x,y)>2\sqrt{t}$, then $d(B_x(\sqrt{t}),B_y(\sqrt{t}))=d(x,y)-2\sqrt{t}$, and hence
\[
-\frac{d^2(B_x(\sqrt{t}),B_y(\sqrt{t}))}{4t}
=-\frac{(d(x,y)-2\sqrt{t})^2}{4t}
\leq -\frac{d^2(x,y)}{4(1+\epsilon)t}+C(\epsilon)
\]
for some constant $C(\epsilon)$, where $\epsilon>0$, and if $\epsilon \to 0$, then the constant $C(\epsilon)\to \infty$.
Therefore, by \eqref{heaup2} we have
\begin{equation}\label{heupp3}
H(x,y,t)\leq\frac{c_7(n,\epsilon)e^{2c_6A}}{V_f(B_x(\sqrt{t})^{1/2}V_f(B_y(\sqrt{t})^{1/2}}
\times\exp\left(-\frac 34\lambda_1t-\frac{d^2(x,y)}{4(1+\epsilon)t}\right)
\end{equation}
for all $x,y\in B_o(\frac 12R)$ and $0<t<R^2/4$. Recall that by Lemma \ref{lecomp1}
\[
V_f(B_x(\sqrt{t}))\leq e^{4A}
\left(\frac{d(x,y)}{\sqrt{t}}+1\right)^n V_f(B_y(\sqrt{t}))
\]
for all $x,y\in B_o(\frac 12R)$ and $0<t<R^2/4$. Therefore we get
\[
H(x,y,t) \leq \frac{c_7(n,\epsilon)e^{(2c_6+2)A}}{V_f(B_x(\sqrt{t})}
\cdot\left(\frac{d(x,y)}{\sqrt{t}}+1\right)^{\frac n2}
\times\exp\left(-\frac 34\lambda_1t-\frac{d^2(x,y)}{(4+\epsilon)t}\right)
\]
for all $x,y\in B_o(\frac 14R)$ and $0<t<R^2/4$.
\end{proof}


\section{$L_f^1$-Liouville theorem}\label{sec4}

In this section, we will prove $L_f^1$-Liouville theorem on complete noncompact smooth metric measure spaces by using the $f$-heat kernel estimates proved in
Section \ref{sec3}. Our result extends the classical $L^1$-Liouville theorem obtained by P. Li \cite{[Li0]} and the weighted versions proved by X.-D. Li \cite{[LD]} and the first author \cite{[Wu2]}.

\vspace{0.5em}

We start from a useful lemma.
\begin{lemma}\label{stocha}
Under the same assumption as in Theorem \ref{Main2}, then the complete smooth metric measure space $(M,g,e^{-f}dv)$ is stochastically complete, i.e.,
\[
\int_MH(x,y,t)e^{-f}dv(y)=1.
\]
\end{lemma}

\begin{proof}
In Lemma \ref{comp}, letting $r_1=R_1=0$, $r_2=1$, $R_2=R>1$ and $x=o$, if $|f|(x)\leq ar^2(x)+b$, then
\[
V_f(B_o(R))\leq C(n,b)R^n e^{c(n,a)R^2}
\]
for all $R>1$. Hence
\begin{equation}\label{integ}
\int^{\infty}_1\frac{R}{\log V_f(B_o(R))}dR=\infty.
\end{equation}
By Grigor'yan's Theorem 3.13 in \cite{[Grig3]}, this implies that the smooth metric measure space $(M,g,e^{-f}dv)$ is stochastically complete.
\end{proof}

Now we prove Theorem \ref{Main2} following the arguments of Li in \cite{[Li0]} . We first prove the following integration by parts formula.

\begin{theorem}\label{liouL1}
Under the same assumption as in Theorem \ref{Main2}, for any nonnegative $L_f^1$-integrable $f$-subharmonic function $u$, we have
\[
\int_M {\Delta_f}_y H(x,y,t)u(y)d\mu(y)=\int_M H(x,y,t)\Delta_fu(y)d\mu(y).
\]
\end{theorem}

\begin{proof}[Proof of Theorem \ref{liouL1}]
Applying Green's theorem to $B_o(R)$, we have
\begin{equation*}
\begin{aligned}
&\left|\int_{B_o(R)}{\Delta_f}_y H(x,y,t)u(y)d\mu(y) -\int_{B_o(R)}H(x,y,t)\Delta_f u(y)d\mu(y)\right|\\
=&\left|\int_{\partial B_o(R)}\frac{\partial}{\partial r} H(x,y,t)u(y)d\mu_{\sigma,R}(y)
-\int_{\partial B_o(R)}H(x,y,t)\frac{\partial}{\partial r}u(y)d\mu_{\sigma,R}(y)\right|\\
\leq& \int_{\partial B_o(R)}|\nabla H|(x,y,t)u(y)d\mu_{\sigma,R}(y) +\int_{\partial B_o(R)}H(x,y,t)|\nabla u|(y)d\mu_{\sigma,R}(y),
\end{aligned}
\end{equation*}
where $d\mu_{\sigma,R}$ denotes the weighted area measure  on $\partial B_o(R)$ induced by $d\mu$. We shall show that the above two boundary integrals vanish as $R\to \infty$. Without loss of generality, we assume $x\in B_o(R/8)$. 

\vspace{0.5em}

\emph{Step 1}. Let $u(x)$ be a nonnegative $f$-subharmonic function. Since $\mathrm{Ric}_f\geq 0$ and $|f|\leq ar^2(x)+b$., by Proposition \ref{PoinDouHarn} we get
\begin{equation} \label{mjbd}
\sup_{B_o(R)}u(x)\leq Ce^{\alpha R^2}V_f^{-1}(2R)\int_{B_o(2R)}u(y)d\mu(y),
\end{equation}
where constants $C$ and $\alpha$ depend on $n$, $a$ and $b$. Let $\phi(y)=\phi(r(y))$ be a nonnegative cut-off function satisfying $0\leq\phi\leq1$,
$|\nabla\phi|\leq\sqrt{3}$ and
\begin{equation*}
\phi(r(y))=\left\{ \begin{aligned}
&1&&\mathrm{on}\,\,B_o(R+1)\backslash B_o(R),\\
&0&&\mathrm{on}\,\,B_o(R-1)\cup(M\backslash B_o(R+2)).\\
\end{aligned} \right.
\end{equation*}
Since $u$ is $f$-subharmonic, by the Cauchy-Schwarz inequality we have
\begin{equation*}
\begin{aligned}
0\leq\int_M\phi^2u\Delta_f u d\mu
=&-\int_M \nabla(\phi^2u)\nabla u d\mu\\
=&-2\int_M \phi u\langle\nabla\phi\nabla u\rangle d\mu -\int_M \phi^2|\nabla u|^2d\mu\\
\leq&2\int_M |\nabla\phi|^2u^2d\mu -\frac12\int_M \phi^2|\nabla u|^2d\mu.
\end{aligned}
\end{equation*}
By \eqref{mjbd}, we have that
\begin{equation*}
\begin{aligned}
\int_{B_o(R+1)\backslash B_o(R)}|\nabla u|^2d\mu\leq& 4\int_M|\nabla \phi|^2u^2d\mu\\
\leq&12\int_{B_o(R+2)}u^2d\mu\\
\leq&12\sup_{B_o(R+2)}u\cdot\|u\|_{L^1(\mu)}\\
\leq&\frac{Ce^{\alpha(R+2)^2}}{V_f(2R+4)}\cdot\|u\|_{L^1(\mu)}^2.
\end{aligned}
\end{equation*}
On the other hand, the Cauchy-Schwarz inequality implies that
\[
\int_{B_o(R+1)\backslash B_o(R)}|\nabla u|d\mu\leq
\left(\int_{B_o(R+1)\backslash B_o(R)}|\nabla u|^2d\mu\right)^{1/2}\cdot [V_f(R+1)\backslash V_f(R)]^{1/2}.
\]
Combining the above two inequalities, we have
\begin{equation}\label{guji}
\int_{B_o(R+1)\backslash B_o(R)}|\nabla u|d\mu\leq C_1e^{\alpha R^2}\cdot\|u\|_{L^1(\mu)},
\end{equation}
where $C_1=C_1(n,a,b)$.

\vspace{0.5em}

\emph{Step 2}. By letting $\epsilon=1$ in Theorem \ref{Main1},
the $f$-heat kernel $H(x,y,t)$ satisfies
\begin{equation}\label{guji2}
H(x,y,t) \leq \frac{c_3}{V_f(B_x(\sqrt{t})}\cdot\left(\frac{d(x,y)}{\sqrt{t}}+1\right)^{\frac n2}
\times\exp\left(c_4R^2-\frac{d^2(x,y)}{5t}\right)
\end{equation}
for any $x,y\in B_o(R)$ and $0<t<R^2/4$, where $c_3=c_3(n,b)$ and $c_4=c_4(n,a)$. Together with \eqref{guji} we get
\begin{equation*}
\begin{aligned}
J_1:=&\int_{B_o(R+1)\backslash B_o(R)}H(x,y,t)|\nabla u|(y)d\mu(y)\\
\leq&\sup_{y\in {B_o(R+1)\backslash B_o(R)}}H(x,y,t)\cdot\int_{B_o(R+1)\backslash B_o(R)}|\nabla u|d\mu\\
\leq&\frac{C_2\|u\|_{L^1(\mu)}}{V_f(B_x(\sqrt{t}))}\cdot\left(\frac{R+1+d(o,x)}{\sqrt{t}}+1\right)^{\frac n2}
\times\exp\left(-\frac{(R-d(o,x))^2}{5t}+c_4(R+1)^2\right),
\end{aligned}
\end{equation*}
where $C_2=C_2(n,a,b)$.

Thus, for $T$ sufficiently small and for all $t\in (0,T)$ there exists a fixed constant $\beta>0$ such that
\[
J_1 \leq \frac{C_3\|u\|_{L^1(\mu)}}{V_f(B_x(\sqrt{t}))}\cdot\left(\frac{R+1+d(o,x)}{\sqrt{t}}+1\right)^{\frac n2}
\times\exp\left(-\beta R^2+c\frac{d^2(o,x)}{t}\right),
\]
where $C_3=C_3(n,a,b)$. Hence for all $t\in(0,T)$ and all $x\in M$, $J_1$ tends to zero as $R$ tends to infinity.

\vspace{0.5em}

\emph{Step 3}.
Consider the integral with respect to $d\mu$,
\begin{equation*}
\begin{aligned}
\int_M\phi^2(y)|\nabla H|^2(x,y,t)
&=-2\int_M\big\langle H(x,y,t)\nabla\phi(y),\phi(y)\nabla H(x,y,t)\big\rangle\\
&\quad-\int_M\phi^2(y)H(x,y,t)\Delta_f H(x,y,t)\\
&\leq2\int_M|\nabla\phi|^2(y)H^2(x,y,t)+\frac 12\int_M\phi^2(y)|\nabla H|^2(x,y,t)\\
&\quad-\int_M\phi^2(y)H(x,y,t)\Delta_f H(x,y,t).
\end{aligned}
\end{equation*}
This implies
\begin{equation}
\begin{aligned}\label{cutfuest}
&\int_{B_o(R+1)\backslash B_o(R)}|\nabla H|^2\\
&\leq\int_M\phi^2(y)|\nabla H|^2(x,y,t)\\
&\leq4\int_M|\nabla\phi|^2H^2-2\int_M\phi^2H\Delta_f H\\
&\leq12\int_{B_o(R+2)\backslash B_o(R-1)}H^2+2\int_{B_o(R+2)\backslash B_o(R-1)}H|\Delta_f H|\\
&\leq12\int_{B_o(R+2)\backslash B_o(R-1)}H^2 +2\left(\int_{B_o(R+2)\backslash B_o(R-1)}H^2\right)^{\frac 12}
\left(\int_M(\Delta_f H)^2\right)^{\frac 12}.
\end{aligned}
\end{equation}
By Lemma \ref{stocha}, we have
\[
\int_MH(x,y,t)d\mu(y)=1
\]
for all $x\in M$ and $t>0$. By \eqref{guji2} we get
\begin{equation}
\begin{aligned}\label{cutfuest2}
\int_{B_o(R+2)\backslash B_o(R-1)}H^2(x,y,t)d\mu
&\leq \sup_{y\in B_o(R+2)\backslash B_o(R-1)}H(x,y,t)\\
&\leq\frac{c_3}{V_f(B_x(\sqrt{t})}\cdot\left(\frac{R+2-d(o,x)}{\sqrt{t}}+1\right)^{\frac n2}\\
&\quad\times\exp\left[-\frac{(R-1-d(o,x))^2}{5t}+c_4(R+2)^2\right].
\end{aligned}
\end{equation}

We \emph{claim} that there exists a constant $C_4>0$ such that
\begin{equation}\label{Lapest}
\int_M(\Delta_f H)^2(x,y,t)d\mu\leq\frac{C_4}{t^2}H(x,x,t).
\end{equation}

Because $f$-heat kernel on $M$ can be obtained by taking the limit of $f$-heat kernels on a compact exhaustion of $M$, it suffices to prove the claim for  $f$-heat kernel on any compact subdomain of $M$.  Let $H(x,y,t)$ is a $f$-heat kernel on a compact subdomain $\Omega\subset M$, by the eigenfunction expansion, we have
\[
H(x,y,t)=\sum^{\infty}_ie^{-\lambda_it}\psi_i(x)\psi_i(y),
\]
where $\{\psi_i\}$ are orthonormal basis of the space of $L_f^2$ functions with Dirichlet boundary value satisfying the equation
\[
\Delta_f\psi_i=-\lambda_i\psi_i.
\]
Differentiating with respect to the variable $y$, we have
\[
\Delta_fH(x,y,t)=-\sum^{\infty}_i\lambda_ie^{-\lambda_it}\psi_i(x)\psi_i(y).
\]
Notice that $s^2e^{-2s}\leq C_5e^{-s}$ for all $0\leq s<\infty$, therefore
\[
\int_M(\Delta_fH)^2d\mu(y)\leq C_5t^{-2}
\sum^{\infty}_ie^{-\lambda_it}\psi^2_i(x)
=C_5t^{-2} H(x,x,t)
\]
and claim \eqref{Lapest} follows.
\vspace{0.5cm}

Combining \eqref{cutfuest}, \eqref{cutfuest2} and \eqref{Lapest}, we obtain
\begin{equation*}
\begin{aligned}
&\int_{B_o(R+1)\backslash B_o(R)}|\nabla H|^2d\mu
\leq C_6\left[V^{-1}_f+t^{-1}V^{-\frac 12}_f H^{\frac 12}(x,x,t)\right]\\
&\quad\quad\quad\quad\times\left(\frac{R+2-d(o,x)}{\sqrt{t}}+1\right)^{\frac n2}
\times\exp\left[-\frac{(R-1-d(o,x))^2}{10t}+c_4(R+2)^2\right].
\end{aligned}
\end{equation*}
where $V_f=V_f(B_x(\sqrt{t}))$. By the Cauchy-Schwarz inequality we get,
\begin{equation}
\begin{aligned}\label{gujiest2}
&\int_{B_o(R+1)\backslash B_o(R)}|\nabla H|d\mu\\
&\leq\left[V_f(B_o(R+1))\backslash V_f(B_o(R))\right]^{1/2}
\times\left[\int_{B_o(R+1)\backslash B_o(R)}|\nabla H|^2d\mu\right]^{1/2}\\
&\leq C_6V^{1/2}_f(B_o(R+1))\left[V^{-1}_f+t^{-1}V^{-\frac 12}_f H^{\frac 12}(x,x,t)\right]^{1/2}\\
&\quad\times\left(\frac{R+2-d(o,x)}{\sqrt{t}}+1\right)^{\frac n4}\times\exp\left[-\frac{(R-1-d(o,x))^2}{20t}+\frac{c_9}{2}(R+2)^2\right].
\end{aligned}
\end{equation}

Therefore, by \eqref{mjbd} and \eqref{gujiest2}, by Cauchy-Schwarz inequality we have
\begin{equation*}
\begin{aligned}
J_2:=&\int_{B_o(R+1)\backslash B_o(R)}|\nabla H(x,y,t)|u(y)d\mu(y)\\
\leq&\sup_{y\in B_o(R+1)\backslash B_o(R)}u(y)\cdot\int_{B_o(R+1)\backslash B_o(R)}|\nabla H(x,y,t)|d\mu(y)\\
\leq& \frac{C_7\|u\|_{L^1(\mu)}}{V^{1/2}_f(B_o(2R+2))}
\cdot\left[V^{-1}_f+t^{-1}V^{-\frac 12}_f H^{\frac 12}(x,x,t)\right]^{1/2}\\
&\times\left(\frac{R+2-d(o,x)}{\sqrt{t}}+1\right)^{\frac n4}\times\exp\left[-\frac{(R-1-d(o,x))^2}{20t}+c_{10}(R+2)^2\right],
\end{aligned}
\end{equation*}
where $V_f=V_f(B_x(\sqrt{t}))$. Similar to the case of $J_1$, by choosing $T$ sufficiently small, for all $t\in(0,T)$ and all $x\in M$, $J_2$ also tends to zero when $R$ tends to infinity.

\vspace{0.5em}

\emph{Step 4}.
By the mean value theorem, for any $R>0$ there exists $\bar{R}\in (R,R+1)$ such that
\begin{equation*}
\begin{aligned}
J:&=\int_{\partial B_o(\bar{R})}\left[H(x,y,t)|\nabla u|(y)+|\nabla H|(x,y,t)u(y)\right]d\mu_{\sigma,\bar{R}}(y)\\
&=\int_{B_o(R+1)\backslash B_p(R)}\left[H(x,y,t)|\nabla u|(y)+|\nabla H|(x,y,t)u(y)\right]d\mu(y)\\
&=J_1+J_2.
\end{aligned}
\end{equation*}
By step 2 and step 3, we know that by choosing $T$ sufficiently small, for all $t\in(0,T)$ and all $x\in M$, $J$ tends to zero as $\bar{R}$ (and hence $R$) tends to infinity. Therefore we finish the proof of Theorem \ref{liouL1} for $T$ sufficiently small.

\vspace{0.5em}

\emph{Step 5}. Using the semigroup property of the $f$-heat equation,
\begin{equation*}
\begin{aligned}
\frac{\partial}{\partial(s+t)}\left(e^{(s+t)\Delta_f}u\right)&=
\frac{\partial}{\partial t}\left(e^{s\Delta_f}e^{t\Delta_f}u\right)=
e^{s\Delta_f}\frac{\partial}{\partial t}\left(e^{t\Delta_f}u\right)\\
&=e^{s\Delta_f}e^{t\Delta_f}(\Delta_fu)
=e^{(s+t)\Delta_f}(\Delta_fu),
\end{aligned}
\end{equation*}
we prove Theorem \ref{liouL1} for all time $t>0$.
\end{proof}

Next we prove the $L_f^1$ Liouville theorem, Theorem \ref{Main2}.

\begin{proof}[Proof of Theorem \ref{Main2}]
Let $u(x)$ be a nonnegative, $L_f^1$-integrable and $f$-subharmonic function defined on $M$. We define a space-time function
\[
u(x,t)=\int_MH(x,y,t)u(y)d\mu(y)
\]
with initial data $u(x,0)=u(x)$. From Theorem \ref{liouL1}, we conclude that
\begin{equation}
\begin{aligned}\label{increfun}
\frac{\partial}{\partial t}u(x,t)
&=\int_M\frac{\partial}{\partial t}H(x,y,t)u(y)d\mu(y)\\
&=\int_M{\Delta_f}_yH(x,y,t)u(y)d\mu(y)\\
&=\int_MH(x,y,t){\Delta_f}_yu(y)d\mu(y)\geq 0,
\end{aligned}
\end{equation}
that is, $u(x,t)$ is increasing in $t$. By Lemma \ref{stocha},
\[
\int_MH(x,y,t)d\mu(y)=1
\]
for all $x\in M$ and $t>0$. So we have
\[
\int_Mu(x,t)d\mu(x)=\int_M\int_MH(x,y,t)u(y)d\mu(y)d\mu(x)=\int_M u(y)d\mu(y).
\]
Since $u(x,t)$ is increasing in $t$, so $u(x,t)=u(x)$ and hence $u(x)$ is a nonnegative $f$-harmonic function,
i.e. $\Delta_fu(x)=0$.

On the other hand, for any positive constant $a$, let us define a new
function $h(x)=\min\{u(x),a\}$. Then $h$ satisfies
\[
0\leq h(x)\leq u(x), \quad|\nabla h|\leq |\nabla u|\quad \mathrm{and}\quad\Delta_f h(x)\leq 0.
\]
In particular, $h$ satisfies estimates \eqref{mjbd} and \eqref{guji}. Similarly we define $h(x,t)$ and
\begin{equation*}
\begin{aligned}
\frac{\partial}{\partial t}h(x,t)&=\frac{\partial}{\partial t}\int_MH(x,y,t)h(y)d\mu(y)\\
&=\int_MH(x,y,t){\Delta_f}_yh(y)d\mu(y)\leq 0.
\end{aligned}
\end{equation*}
By the same argument, we have that $\Delta_fh(x)=0$.

By the regularity theory of $f$-harmonic functions, this is impossible unless $h=u$ or $h=a$. Since $a$ is arbitrary and $u$ is nonnegative, so $u$ must be identically constant. The theorem then follows from the fact that the absolute value of a $f$-harmonic function is a nonnegative $f$-subharmonic function.
\end{proof}


\section{$L_f^1$-uniqueness property}\label{sec5}

For the completeness we provide a detailed proof of Theorem \ref{Main3} following the arguments of Li in \cite{[Li0]}.

\begin{proof}[Proof of Theorem \ref{Main3}]
Let $u(x,t)\in L_f^1$ be a nonnegative function satisfying the assumption in Theorem \ref{Main3}.
For $\epsilon>0$, let $u_\epsilon(x)=u(x,\epsilon)$. Define
\begin{equation}\label{semgro1}
e^{t\Delta_f}u_\epsilon(x)=\int_MH(x,y,t)u_\epsilon(y)d\mu(y)
\end{equation}
and
\[
F_\epsilon(x,t)=\min\{0,u(x,t+\epsilon)-e^{t\Delta_f}u_\epsilon(x)\}.
\]
Then $F_\epsilon(x,t)$ is nonnegative and satisfies
\[
\lim_{t\to 0}F_\epsilon(x,t)=0\quad \mathrm{and}
\quad(\partial_t-\Delta_f)F_\epsilon(x,t)\leq0.
\]
Let $T>0$ be fixed. Let $h(x)=\int^T_0F_\epsilon(x,t)dt$,
which satisfies
\begin{equation}\label{condst1}
\begin{split}
\Delta_fh(x)=&\int^T_0\Delta_fF_\epsilon(x,t)dt\\
\geq&\int^T_0\partial_tF_\epsilon(x,t)dt=F_\epsilon(x,T)\geq0.
\end{split}
\end{equation}
Moreover,
\begin{equation*}
\begin{aligned}
\int_Mh(x)d\mu&=\int^T_0\int_MF_\epsilon(x,t)d\mu dt\\
&\leq\int^T_0\int_M|u(x,t+\epsilon)-e^{t\Delta_f}u_\epsilon(x)|d\mu dt\\
&\leq\int^T_0\int_Mu(x,t+\epsilon)d\mu dt+\int^T_0\int_Me^{t\Delta_f}u_\epsilon(x)d\mu dt<\infty,
\end{aligned}
\end{equation*}
where the first term on the right hand side is finite from our assumption, and the second term is finite because $e^{t\Delta_f}$ is a contractive semigroup
in $L_f^1$. Therefore, $h(x)$ is a nonnegative $L_f^1$-integrable $f$-subharmonic function. By Theorem \ref{Main2}, $h(x)$ must be constant. Combining with \eqref{condst1} we have $F_\epsilon(x,t)=0$. Hence $F_\epsilon(x,T)\equiv0$ for all $x\in M$ and $T>0$, which implies
\begin{equation}\label{tianj1}
e^{t\Delta_f}u_\epsilon(x)\geq u(x,t+\epsilon).
\end{equation}

Next we estimate the function $e^{t\Delta_f}u_\epsilon(x)$ in \eqref{semgro1}. Applying the upper bound estimate of the heat kernel $H(x,y,t)$ and letting $R=2d(x,y)+1$, we have
\begin{equation*}
\begin{aligned}
e^{t\Delta_f}u_\epsilon(x)&\leq\frac{C}{V_f(B_x(\sqrt{t})}
\cdot\left(\frac{d(x,y)}{\sqrt{t}}+1\right)^{\frac n2}\\
&\quad\times\int_M \left[\exp\left(\tilde{C}d^2(x,y)-\frac{d^2(x,y)}{5t}\right)u(y,\epsilon)\right]d\mu(y).
\end{aligned}
\end{equation*}
Thus there exists a sufficiently small $t_0>0$ such that for all $0<t<t_0$, we have $\lim_{\epsilon \to0}e^{t\Delta_f}u_\epsilon(x)=0$ by the assumption
\[
\lim_{\epsilon\to 0}\int_Mu(x,\epsilon)d\mu(x)=0.
\]

Therefore by the semigroup property, we conclude that $\lim_{\epsilon \to0}e^{t\Delta_f}u_\epsilon(x)=0$ for all $x\in M$ and $t>0$. Combining with \eqref{tianj1} we get $u(x,t)\leq 0$. Therefore $u(x,t)\equiv 0$.
\end{proof}

\section{Appendix}\label{apdix}

In the appendix we solve for the $f$-heat kernel of $1$-dimensional steady Gaussian soliton $(\mathbb{R},\ g_0, e^{-f}dx)$, where $g_0$ is the Euclidean metric, and $f=k x$ with $k=\pm 1$. The method is standard separation of variables. Suppose the $f$-heat kernel is of the form
\[
H(x,y,t)=
\varphi(y)\phi(x)\psi(t)\times\exp\left(-\frac{|x-y|^2}{4t}\right).
\]
For a fixed $y$, we get
\begin{equation*}
\begin{split}
H_t &= \varphi\phi e^{-\frac{|x-y|^2}{4t}}\left(\psi_t+\psi\frac{|x-y|^2}{4t^2}\right),\\
H_x &= \varphi\psi e^{-\frac{|x-y|^2}{4t}}\left(\phi_x-\phi\frac{x-y}{2t}\right),\\
H_{xx} &= \varphi\psi e^{-\frac{|x-y|^2}{4t}}
\left(\phi_{xx}+\phi\frac{|x-y|^2}{4t^2}-\phi_x\frac{x-y}{t}-\phi\frac{1}{2t}\right).
\end{split}
\end{equation*}
So $H_t=H_{xx}-f_xH_x$ implies
\begin{equation*}
\begin{split}
\phi\left(\psi_t+\psi\frac{|x-y|^2}{4t^2}\right) &= \psi
\left(\phi_{xx}+\phi\frac{|x-y|^2}{4t^2}-\phi_x\frac{x-y}{t}-\frac{\phi}{2t}\right)-k\psi\left(\phi_x-\phi\frac{x-y}{2t}\right).
\end{split}
\end{equation*}
That is,
\[
\frac{\psi_t}{\psi}=
\frac{\phi_{xx}-k\phi_x}{\phi}-\frac{x-y}{2t}\cdot\frac{2\phi_x-k\phi}{\phi}-\frac{1}{2t}.
\]
Therefore
\begin{equation*}
\begin{split}
\frac{\phi_{xx}-k\phi_x}{\phi} &= C_1,\hspace{1cm}
\frac{(2\phi_x-k\phi)(x-y)}{\phi} = C_2,\hspace{1cm}
\frac{\psi_t}{\psi} = C_1-\frac{1+C_2}{2t},
\end{split}
\end{equation*}
From above, their solutions are
\begin{equation*}
\begin{split}
\phi &= C_3 e^{\frac{1}{2}kx},\\
\psi &= C_4 \frac{1}{\sqrt{t}}e^{-4/t},
\end{split}
\end{equation*}
where $C_1$, $C_2$, $C_3$, $C_4$ are constants.

By the initial condition $\lim_{t\to 0}u(x,t)=\delta_{f,y}(x)$ we get $\varphi(y)= e^{\frac{1}{2}ky}$, and $C_3C_4=\frac{1}{2\sqrt{\pi}}$. Therefore the $f$-heat kernel is
\[
H(x,y,t)=\frac{e^{\pm \frac{x+y}{2}}\cdot e^{-t/4}}{(4\pi t)^{1/2}}
\times\exp\left(-\frac{|x-y|^2}{4t}\right).
\]
It is easy to check that $\int_{\mathbb{R}}H(x,y,t)e^{-f(x)}dx=1$, which confirms the stochastic completeness proved in Lemma 4.1.

\qed

\bibliographystyle{amsplain}

\end{document}